\documentclass[10pt]{amsart} 
\RequirePackage[utf8]{inputenc}

\usepackage{amsmath,amsthm,amsfonts,amssymb,amscd,amsbsy,dsfont,hyperref}
\usepackage[all]{xy}

\usepackage{palatino}

\renewcommand{\geq}{\geqslant}
\renewcommand{\leq}{\leqslant}

\usepackage{geometry}
\allowdisplaybreaks

\hypersetup{
    pdftoolbar=true,
    pdfmenubar=true,
   pdffitwindow=false,
    pdfstartview={FitH},
    pdftitle={}, 
    pdfauthor={}, 
    pdfsubject={}, 
    pdfkeywords={},
    pdfnewwindow=true,
    colorlinks=true, 
    linkcolor=blue,
    citecolor=blue,
    urlcolor=black,}
\newtheorem*{acknowledgement}{Acknowledgement}
\newtheorem{corollary}{Corollary}
\newtheorem{conjecture}{Conjecture}

\newtheorem{lemma}{Lemma}
\newtheorem{proposition}{Proposition}

\newtheorem{theorem}{Theorem}

\numberwithin{equation}{section}

\begin{document}
\title[Minimal Volume and Minimal Curvature on 4-manifolds]{Estimates for Minimal Volume and Minimal Curvature\\ on 4-dimensional compact manifolds}
\author{E. Costa, R. Di\'{o}genes \& E. Ribeiro Jr.}
\address[E. Costa]{Universidade Federal da Bahia - UFBA, Departamento de Matem\'{a}tica, Campus de Ondina, Av. Ademar de Barros, 40170-110-Salvador / BA, Brazil.}\email{ezio@ufba.br}
\address[R. Di\'ogenes]{Universidade Federal do Cear\'a - UFC, Departamento  de Matem\'atica, Campus do Pici, Av. Humberto Monte, Bloco 914,
60455-760-Fortaleza / CE , Brazil.} \email{rafaeljpdiogenes@gmail.com}
\address[E. Ribeiro Jr]{Current: Department of Mathematics, Lehigh University, Bethlehem - PA, 18015, United States\\ Permanent: Universidade Federal do Cear\'a - UFC, Departamento  de Matem\'atica, Campus do Pici, Av. Humberto Monte, Bloco 914,
60455-760-Fortaleza / CE , Brazil.}
 \email{ernani@mat.ufc.br}
\thanks{R. Di\'ogenes acknowledges partial support by FUNCAP/Brazil}
\thanks{E. Ribeiro Jr acknowledges partial support by grants from FUNCAP/Brazil and CNPq/Brazil}
\date{June 28, 2014}

\newcommand{\spacing}[1]{\renewcommand{\baselinestretch}{#1}\large\normalsize}
\spacing{1.2}

\begin{abstract}
In a remarkable article published in 1982, M. Gromov introduced the concept of minimal volume, namely, the minimal volume of a manifold $M^n$ is defined to be the greatest lower bound of the total volumes of $M^n$ with respect to complete Riemannian metrics whose sectional curvature is bounded above in absolute value by 1. While the minimal curvature, introduced by G. Yun in 1996, is the smallest pinching of the sectional curvature among metrics of volume 1. The goal of this article is to provide estimates to minimal volume and minimal curvature on 4-dimensional compact manifolds involving some differential and topological invariants. Among these ones, we get some sharp estimates for minimal curvature. 
\end{abstract}

\maketitle

\section{Introduction}\label{introduction}
Let $M^n$ be a n-dimensional compact oriented smooth manifold and $\mathcal{M}$ the set of smooth Riemannian structures on $M^n.$ We consider all complete Riemannian structures $g\in\mathcal{M}$ whose sectional curvatures satisfy $|K(g)|\leq 1.$ Under these notations, Gromov \cite{Gromov} introduced the concept of minimal volume. More precisely, the {\it minimal volume} of $M^n$ is defined by

\begin{equation}
\label{minvol}
{\rm Min\,Vol(M)}=\inf_{ |K(g)|\leq 1} {\rm Vol(M,g)}.
\end{equation}

This concept is closely related with others important invariants. For instance, Paternain and Petean  \cite{PP} proved that the minimal volume, {\it minimal entropy} ${\rm h(M)}$ and {\it simplicial volume} $\| M\|,$ on a compact manifold $M^n,$ are related as follows
\begin{equation}
c(n)\| M\|\leq [{\rm h(M)}]^n \leq (n-1)^{n} {\rm Min\,Vol(M)},
\end{equation}where $c(n)$ is a positive constant; for more details see \cite{LS} and \cite{PP}.

In \cite{Bessieres}, Bessieres proved that the value of the minimal volume may depend on the differentiable structures of $M^n.$ In particular, he gave examples of high-dimensional manifolds which are homeomorphic, but have different positive minimal volumes; see also \cite{Kotschick} and \cite{Manning}.

Besson, Courtois and Gallot \cite{BCG1} shed light on the following problems:

\begin{flushright}
\begin{minipage}[t]{4.37in}
 \begin{itemize}
\item \emph{When the minimal volume is identically zero?}
\item \emph{When it is, what can we say of a such manifold?}
\item \emph{And, when a manifold has minimal volume positive?}
\item \emph{Does there exists a metric that realizes the minimum?}
\end{itemize} [see also \cite{bergerB}, Question 266.]
 \end{minipage}
\end{flushright}

This subject has received a lot of attention. In the last decades many authors have been proved useful results on this subject. Among them, we detach the next ones: Gromov \cite{Gromov} proved in 1982 that if a compact manifold $M^n$ admits a metric of negative sectional curvature, then the minimal volume of $M^n$ is positive. Moreover, it is well-known that if $M^n$ is compact admitting a flat metric, then ${\rm Min\,Vol(M)}=0.$ Also, from Gauss-Bonnet formula if $M^2$ is a compact oriented surface, then ${\rm Min\,Vol(M)}\geq2\pi|\chi(M)|,$ where $\chi(M)$ is the Euler characteristic of $M^2,$ with equality if the Gauss curvature $K$ is constant equal to $1$ or $-1.$ We highlight that the torus and the Klein bottle have zero minimal volume. Indeed, they support some flat metrics, but there is not metric realizing the minimal volume. While Wang and Xu \cite{WX} were able to show that the minimal volume of $\mathbb{S}^{2n+1}$ and $\mathbb{S}^{n}\times\mathbb{R},$ for $n\geq 1,$ are identically zero. Gromov \cite{Gromov} also proved that if $M^n$ is compact, then ${\rm Min\,Vol(M)}\geq c(n)|\chi(M)|,$ for some positive constant $c(n)$ depending on the dimension. In the same direction Cheeger and Gromov proved  that if $M^n$ admits a polarized $\mathcal{F}$-structure, then the minimal volume of $M^n$ must be zero (cf. \cite{CG} and  \cite{Gromov}). Moreover, they showed that if the $\mathcal{F}$-structure has positive rank, then the Euler characteristic of $M^n$ vanishes. In \cite{PP}, Paternain and Petean also proved that if $M^n$ admits an $\mathcal{F}$-structure, then it collapses with curvature bounded from below, in other words, there exists a sequence of metric $g_i$ for which the sectional curvature is uniformly bounded from below, but their volumes approach to zero as $i$ goes to infinity. Furthermore, from Rong's work \cite{rong}, in dimension 4, small minimal volume implies zero minimal volume. For comprehensive references on such a theory, we address to \cite{BCG1}, \cite{CG}, \cite{Gromov} and \cite{PP}.

We recall that a compact Riemannian manifold $(M^n,\, g_0)$ is called {\it hyperbolic} if the universal covering of $M^n$ is isometric to hyperbolic space $\mathbb{H}^n$ and in this case $g_0$ is called a hyperbolic metric. Any orientable surface with genus bigger than one admits a metric with constant negative sectional curvature and then such a metric is hyperbolic.

In 1982, Gromov \cite{Gromov} posed the following conjecture.

\begin{conjecture}[Gromov]
\label{conjecture1}
Let $(M^n,\,g_0)$ be a complete hyperbolic manifold with finite vo\-lume. Then ${\rm Min\,Vol(M)}=Vol(M, g_0).$
\end{conjecture}

In \cite{BCG2}, Besson, Courtois and Gallot  gave a partial answer to Conjecture \ref{conjecture1}.  More precisely, they proved that on a compact hyperbolic manifold, the minimal volume is achieved by the hyperbolic metric.

In order to proceed we recall the concept of \textit{minimal curvature} introduced by Yun in \cite{Yun}. Let $\mathcal{M}_1$ be the set of smooth Riemannian structures on $M^4$ of volume 1. For $g\in\mathcal{M}_1,$ we define the $L^{\infty}$-functional by $$\mathcal{R}^{\infty}(g)=|Rm(g)|_{\infty},$$ where $Rm$ denotes the Riemann curvature tensor of the metric $g$ and $|\cdot|_{\infty}$ denotes the sup-norm. Then we define the \textit{minimal curvature} by
\begin{equation}
\label{mincur}Mincur(M)=\inf_{g\in\mathcal{M}_1}\mathcal{R}^{\infty}(g).
\end{equation} The minimal curvature is also related with the minimal volume. For instance, Yun \cite{Yun} showed that the minimal curvature is zero if and only if the minimal volume is zero (cf. Lemma 2.1 in the quoted article).

One should point out that 4-dimensional  manifolds have special behavior. In large part, this is because the bundle of $2$-forms on a 4-dimensional compact oriented Riemannian manifold can be invariantly decomposed as a direct sum; some relevant facts may be found in \cite{atiyah}, \cite{besse} and \cite{scorpan}. In this paper, we are interested in to investigate the minimal volume and minimal curvature of 4-dimensional compact manifolds. More precisely, based on the ideas developed in \cite{BCG3}, \cite{CR}, \cite{Gromov} and \cite{Yun}, we shall use the concepts of  biorthogonal (sectional) curvature to provide some estimates to minimal volume and minimal curvature on 4-dimensional compact manifolds involving some differential and topological invariants. In what follows $M^4$ will denote a compact oriented 4-dimensional manifold and $g$ is a Riemannian metric on $M^4$ with  scalar curvature $s_g,$ or simply $s,$ and  sectional curvature $K.$ Furthermore, we denote by $\chi(M^4)$ the Euler characteristic of $M^4.$ While the signature of $M^4$ is denoted by $\tau(M^4),$ which is, in module, a topological invariant.

In order to set the stage for the results to follow let us recall briefly the concept of biorthogonal curvature. For each plane $P\subset T_{p}M$ at a point $p\in M^4,$  we define  {\it the bi\-or\-tho\-gonal (sectional) curvature} of $P$ by the following average of the sectional curvatures
\begin{equation}\label{[1.2]}
 \displaystyle{K^\perp (P) = \frac{K(P) + K(P^\perp) }{2}},
\end{equation}
where $P^\perp$  is the orthogonal plane to  $P.$ 

The sum of two sectional curvatures on two orthogonal planes, which was perhaps first observed by Chern \cite{Chern}, plays a very crucial role on four-dimensional manifolds and it is of fundamental importance for our purposes here. This notion also appeared in works due to LeBrun \cite{LeBrun}, Noronha \cite{Noronha} and Seaman \cite{SeamanT}. Surprisingly, $\Bbb{S}^{1}\times \Bbb{S}^{3}$ with its canonical metric shows that the positivity of the biorthogonal curvature is an intermediate condition between positive sectional curvature and positive scalar curvature. Moreover, a 4-dimensional Riemanniana manifold $(M^4,\,g)$ is Einstein if and only if $K^\perp(P)=K(P)$ for any plane $P\subset T_{p}M$ at any point $p\in M^4$ (cf. Corollary 6.26 \cite{handbook} and \cite{ST}). For more details see \cite{renato}, \cite{CR}, \cite{Noronha}, \cite{noronha2}, \cite{SeamanT} and \cite{Seaman}.

As we have pointed out dimension four enjoys a privileged satus. For instance, on an oriented Riemannian manifold $(M^4,\,g),$ the Weyl curvature tensor $W$ is an endomorphism of the bundle of 2-forms $\Lambda^2=\Lambda^{2}_{+}\oplus\Lambda^{2}_{-}$ such that  $W = W^+\oplus W^-,$ where    $W^\pm : \Lambda^{2}_\pm \longrightarrow \Lambda^{2}_\pm$ are called of the self-dual and anti-self-dual parts of $W.$ We then fix a point and diagonalize $W^\pm$ such that $w_i^\pm,$ $1\le i \le 3,$ are their respective eigenvalues. In particular, they satisfy
\begin{equation}
\label{eigenvalues}
w_1^{\pm}\leq w_2^{\pm}\leq w_3^{\pm}\,\,\,\,\hbox{and}\,\,\,\,w_1^{\pm}+w_2^{\pm}+w_3^{\pm} = 0.
\end{equation}

For our purposes we recall that, as it was explained in \cite{CR} and \cite{Seaman}, the definition of  biorthogonal curvature provides the following useful identities
\begin{equation}
\label{[1.6]}
K_1^\perp = \frac{w_1^+ + w_1^-}{2}+\frac{s}{12}
\end{equation} and
\begin{equation}\label{[1.8]}
K_3^\perp = \frac{w_3^+ + w_3^-}{2}+\frac{s}{12},
\end{equation} where $K_1^\perp(p) = \textmd{min} \{K^\perp(P); P\subset T_{p}M \}$ and $K_3^\perp(p) = \textmd{max}\{K^\perp (P); P\subset T_{p}M \}.$

We also recall that, based in a work due to Lebrun and Gursky \cite{GLB}, Cheng and Zhu \cite{CZ} as well as Itoh \cite{Itoh} studied the modified Yamabe problem in terms of a functional depending on Weyl curvature tensor (see also Section 2.2 in \cite{listing10}).  For the sake of completeness let us briefly outline this construction. Firstly, we consider applications $f:\mathcal{M}\to C^{0,\alpha}(M)$ depending on Weyl conformal curvature tensor $W_{g}$ of $(M^4,\,g)$ satisfying the following conditions:
\begin{eqnarray}
\label{F}
\left \{ \begin{array}{ll}
       f(W)\geq 0,\\
f(\overline{W})=u^{-2}f(W),
                  \end{array} \right.
\end{eqnarray} where $\overline{W}$ denotes the Weyl curvature tensor of $(M^4,\overline{g})$ and $\bar{g}=u^2g.$ This allows us to define the {\it modified Yamabe functional}
\begin{equation}
\label{yamabemodified}
  \displaystyle{\mathcal{Y}^{f}(M^4,\,g)=\frac{1}{Vol(M,g)^{\frac{1}{2}}}\int_{M}\big(s_{g} -f(W)\big)dV_g},
\end{equation} where, in this case, $s-f(W)$ is called {\it modified scalar curvature} of $(M^4,g).$ In fact, from \cite{CRS} it is not difficult to prove that $f_1(W)=-6(w_{1}^{+}+w_{1}^{-})$ satisfies (\ref{F}) and we can use (\ref{[1.6]}) to deduce
\begin{equation}
\label{K1F1}s-f_1(W)=12K_{1}^{\perp}.
\end{equation} From this, $12K_1^{\perp}$ is a modified scalar curvature. In particular, we obtain
\begin{equation}
\label{defY1}\mathcal{Y}_{1}^{\perp}(M,[g])=\inf_{\overline{g}\in[g]}\left\{\frac{12}{Vol(M,g)^{\frac{1}{2}}}\int_M\overline{K}_{1}^{\perp}dV_{\overline{g}}\right\}.
\end{equation} From here it follows that the {\it modified Yamabe invariant} $\mathcal{Y}_{1}^{\perp}(M)$ is given by
\begin{equation}
\label{Y1}\mathcal{Y}_{1}^{\perp}(M)=\sup_{g\in\mathcal{M}} \mathcal{Y}_{1}^{\perp}(M,[g]).
\end{equation} For more details on this construction see \cite{CZ}, \cite{CRS} and \cite{Itoh}.

Here, we introduce 
\begin{equation}
\label{VolKperp}{\rm Vol_{|K^{\perp}|}(M)}=\inf \{{\rm Vol(M^4,g)};\,|K^{\perp}|\leq 1\}.
\end{equation} From (\ref{[1.2]}), it turns out that ${\rm Min\,Vol(M)}\geq{\rm Vol_{|K^{\perp}|}(M)}.$

Similar to definition (\ref{VolKperp}) we use (\ref{[1.6]}) to define
\begin{equation}
\label{VolK1perp}{\rm Vol_{|K_{1}^{\perp}|}(M)}=\inf \{{\rm Vol(M^4,g)};\,|K_{1}^{\perp}|\leq 1\}.
\end{equation} Clearly, we have
\begin{equation}
\label{lp1}
{\rm Vol_{|K^{\perp}|}(M)}\geq{\rm Vol_{|K_{1}^{\perp}|}(M)}.
\end{equation}

After these settings we may state our first result.

\begin{theorem}
\label{thMinSigEul}
Let $M^4$ be a 4-dimensional oriented compact manifold. Then the following estimates hold:
\begin{enumerate}
\item ${\rm Min\,Vol(M)}\geq {\rm Vol_{|K^{\perp}|}(M)}\geq\frac{9\pi^2}{20}|\tau(M)|.$\label{eq1thm1}

\item If $\chi(M)>0,$ then ${\rm Min\,Vol(M)}\geq{\rm Vol_{|K^{\perp}|}(M)}\geq\frac{12\pi^2}{25}\chi(M).$ \label{eq2thm1}

\item If $\chi(M)\leq 0,$ then ${\rm Min\,Vol(M)}\geq\frac{4\pi^2}{3}|\chi(M)|.$ \label{eq3thm1}

\item If $\mathcal{Y}_1^\perp(M)\leq 0,$ then ${\rm Min\,Vol(M)}\geq{\rm Vol_{|K^{\perp}|}(M)}\geq{\rm Vol_{|K_{1}^{\perp}|}(M)}\geq\frac{1}{144}|\mathcal{Y}_{1}^\perp(M)|^2.$\label{eq4th1}
\end{enumerate}
\end{theorem}

As it was previously mentioned a result of Gromov \cite{Gromov} asserts that if $M^n$ is any compact manifold, then 
\begin{equation}
\label{g1}
{\rm Min\,Vol(M)}\geq c(n)|\chi(M)|,
\end{equation} where $c(n)$ is some positive constant depending on the dimension. We also highlight that Gromov proved inequalities analogous to (\ref{g1}) for Pontryagin numbers. Our Theorem \ref{thMinSigEul}  says that, in dimension 4, one may take explicitly a value to constant $c$ in (\ref{g1}). Moreover, since the signature is a Pontryagin number we conclude that this same comment applies to the statement involving the signature. Moreover, as an immediate consequence of Theorem \ref{thMinSigEul} we deduce the following corollary which was first pointed out by Berger.

\begin{corollary}\label{corzeroSigEul2}
Let $M^4$ be a 4-dimensional oriented compact Einstein manifold. Then $M^4$ has zero minimal volume if and only if  $M^4$ is flat.
\end{corollary}

In the sequel, as a consequence of  Freedman's work \cite{freedman}  we obtain the following characterization.

\begin{theorem}
\label{thmB}
Let $M^4$ be a 4-dimensional simply connected compact manifold satisfying 
\begin{equation}
\label{hipcor2}
{\rm Vol_{|K^{\perp}|}(M)}\leq \frac{36}{25}\pi^2.
\end{equation} Then $M^4$ is homeomorphic to either $\Bbb{S}^4$ or $\Bbb{CP}^2.$
\end{theorem}

We highlight that the result obtained in Theorem \ref{thmB} even is true replacing ${\rm Vol_{|K^{\perp}|}(M)}$ in (\ref{hipcor2}) by  ${\rm Min\,Vol(M)}.$ We also remark that the minimal volume is not preserved by homeomorphism. In fact, as it was previously mentioned Bessieres \cite{Bessieres} gave examples of manifolds that are homeomorphic but have different minimal volumes. Furthermore, there are pairs of homeomorphic 4-dimensional manifolds for which the minimal volume is zero for one, and is positive for the other. However, such examples can not be simply connected; for more details see \cite{Kotschick}.

Now, we present some estimates to minimal curvature involving  others topological invariants. More precisely, we have the following result.

\begin{theorem}\label{thmincur}
Let $M^4$ be a 4-dimensional oriented compact manifold. Then the following estimates hold:
\begin{enumerate}
\item $Mincur(M)\geq 2\pi\sqrt{2|\chi(M)|}.$
\item If $\mathcal{Y}(M)\ge 0,$ then $Mincur(M)\geq 2\pi\sqrt{3|\tau(M)|}.$
\item If $\mathcal{Y}(M)\le 0,$ then $Mincur(M)\geq \sqrt{12\pi^2|\tau(M)|+\frac{|\mathcal{Y}(M)|^{2}}{24}}.$
\item If $\mathcal{Y}_1^\perp(M)\leq0,$ then $Mincur(M)\geq\frac{1}{6\sqrt{2}}|\mathcal{Y}_{1}^{\perp}(M)|.$
\end{enumerate}
\end{theorem}

It should be emphasized that our estimate obtained in the first item of Theorem \ref{thmincur} is sharp. To clarify our claim, let $g_0$ be the standard metric on $\mathbb{S}^4$ and we denote $c^2=Vol(M,\,g_0)=\frac{8\pi^2}{3}.$ From this, we choose $\overline{g}=c^{-1}g_0$ and thus $(\mathbb{S}^4,\overline{g})$ is Einstein and locally conformally flat. Next, it suffices to use the decomposition of the Riemann curvature tensor of $(\mathbb{S}^4,\overline{g})$ to arrive at $|\overline{Rm}|^2=\frac{\overline{s}^2}{24},$ where $\overline{Rm}$ denotes the Riemann curvature tensor on metric $\overline{g}.$ This means that $|\overline{Rm}|^2=\frac{144}{24}c^2=16\pi^2$ and then $|\overline{Rm}|=4\pi.$ In particular, $\mathcal{R}^{\infty}(\overline{g})=4\pi.$ Finally, since $Vol(\mathbb{S}^4,\overline{g})=c^{-2}Vol(\mathbb{S}^4,g_0)=1$  and $\chi(\Bbb{S}^4)=2$ we obtain $Mincur(\mathbb{S}^4)=4\pi,$ which settles our claim. Moreover, arguing in the same way, it is not difficult to prove that the lower bound in the first item of Theorem \ref{thmincur} is also attained by $\Bbb{CP}^2$ and $\Bbb{S}^2 \times \Bbb{S}^2$ up to suitable scaling of the Fubini-Study metric and the product metric, respectively.

As an immediate consequence of  the first estimate stated in Theorem \ref{thmincur} we deduce the following result.

\begin{corollary}
Let $M^4$ be a 4-dimensional simply connected compact manifold. Then we have: 
\begin{equation}
Mincur(M)\geq 4\pi.
\end{equation}
\end{corollary}

Next, since flat torus has zero minimal volume we may use Lemma 2.1 in \cite{Yun} to conclude that it has zero minimal curvature. This tell us that our estimate stated in the second Theorem \ref{thmincur} is attained by flat torus. Nonetheless, it is interesting to find a non-flat metric that attains this lower bound. Moreover, the lower bound obtained in the third item of Theorem \ref{thmincur} is attained by compact complex-hyperbolic 4-manifold $\Bbb{C}\mathcal{H}^{2}/\Gamma.$ In particular, combining Hitchin's work \cite{hitchin} with the second estimate obtained in Theorem \ref{thmincur} we directly obtain the following corollary.

\begin{corollary}
Let $(M^4,\,g)$ be a 4-dimensional oriented compact Einstein manifold satisfying $Mincur(M)= 2\pi\sqrt{3|\tau(M)|}.$ Then $M^4$ is either flat or its universal cover is a $K3$ surface.
\end{corollary}

One question that naturally arises from the previous comments is what occurs when the estimate obtained in the first item of Theorem \ref{thmincur} is actually an equality. In fact, under this condition we have the following characterization.

\begin{theorem}\label{thmD}
Let $M^4$ be a 4-dimensional simply connected compact manifold. We assume that $Mincur(M)= 2\pi\sqrt{2|\chi(M)|}$ and there is a metric $g$ that realizes the minimal curvature, i.e., $Mincur(M)=\mathcal{R}^{\infty}(g).$ Then $M^4$ is Einstein and $|W|$ is constant. In addition, if $M^4$ has nonnegative sectional curvature, then $M^4$ is isometric to $\Bbb{S}^4,$ $\Bbb{CP}^2$ or $\Bbb{S}^2\times \Bbb{S}^2.$
\end{theorem}

It would be interesting to know when 4-dimensional manifolds can carry a metric of nonnegative sectional curvature. Even under Einstein assumption. We highlight that known examples with positive curvature are surprisingly rare. In dimension 4, rotationally elliptic 4-dimensional compact manifolds are homeomorphic to one of $\Bbb{S}^4,$ $\Bbb{CP}^2,$ $\Bbb{S}^{2}\times \Bbb{S}^{2},$ $\Bbb{CP}^2 \sharp\Bbb{CP}^2$ or $\Bbb{CP}^2 \sharp\overline{\Bbb{CP}}^2.$ From this, it has been conjectured that only the first two can admit positive curvature. This question is directly related to the Bott conjecture, which asks if a compact simply connected manifold with nonnegative sectional curvature must be elliptic; for more details in this subject we recommend two fascinating surveys due to Ziller \cite{ziller,ziller2}. Nonetheless, it is possible to exhibit a metric of nonnegative sectional curvature on the connected sum $\Bbb{CP}^2 \sharp\overline{\Bbb{CP}}^2$ (cf. Example 45 in \cite{petersen} p. 212, see also \cite{cheeger}). 

At same time, the Hopf conjecture asks if there exists no metric with positive sectional curvature on $\Bbb{S}^2 \times \Bbb{S}^2.$ Concerning to this problem, since $\chi(\Bbb{S}^2 \times \Bbb{S}^2)=4,$ we may use Theorem \ref{thmD} to deduce the following result.

\begin{corollary}
\label{corD1}  $\Bbb{S}^2 \times \Bbb{S}^2$ does not carry any metric $g$ of both positive sectional curvature and $\mathcal{R}^{\infty}(g) \leq 4\pi \sqrt{2}.$
\end{corollary}

\hspace{0.5cm}

This article is organized as follows. In Section \ref{prel}, we prove some basics results which we shall use here. Moreover, we  prove some lemmas and propositions which will be useful to prove our main results. In Section \ref{proofs}, we prove the main results.

\section{Preliminaries}
\label{prel}
Throughout this section we provide some basic lemmas and propositions that will be useful in the proof of our main results. 

First of all, since $Vol(e^{2f}g)=e^{4f}Vol(g)$ it follows that on a compact 4-dimensional manifold the functional
\begin{equation}
\label{funW}
 \mathcal{W}(g) =\int_{M} \mid W_g\mid^2dV_g
\end{equation} is conformally invariant, i.e. $ \mathcal{W}(e^{2f}g) =\mathcal{W}(g).$ As it was showed by Kobayashi in \cite{Kobayashi}  the invariant
\begin{equation}
\label{Weylinvariant}
\mathcal{W}(M)= \inf \{\mathcal{W}(g);\, g\in \mathcal{M}\}
\end{equation} reflects certain global properties of a manifold. Moreover, it is natural to ask if there is a minimizing metric on $M^4$ which achieves $\mathcal{W}(M).$ Clearly, if $g$ is locally conformally flat, then $\mathcal{W}(M)=0.$

In \cite{CRS}, it was defined the functional $\mathcal{E}_{1}^{\perp}:\mathcal{M}\rightarrow \Bbb{R}$ as follows
\begin{equation}
\label{Efunct}\mathcal{E}_{1}^{\perp}(g)=\int_M(s-12K_{1}^{\perp})^2dV_g,
\end{equation}
which also is conformally invariant (cf. Proposition 2 in  \cite{CRS}). Moreover, it was introduced the invariant
\begin{equation}
\label{Einvariant}\mathcal{E}_{1}^{\perp}(M)=\inf_{g\in\mathcal{M}}\mathcal{E}_{1}^{\perp}(g).
\end{equation}

The relationship between (\ref{funW}) and (\ref{Efunct}) is given by the next lemma.

\begin{lemma}\label{keylemma}
Let $M^4$ be a 4-dimensional oriented compact manifold. Then
\begin{equation*}
\mathcal{W}(g)\leq\frac{1}{6}\mathcal{E}_{1}^{\perp}(g).
\end{equation*}
\end{lemma}
\begin{proof} Since the proof of this lemma is very short, we include it here for the sake of completeness. In fact, from (\ref{eigenvalues}) we deduce $w_{1}^{\pm}\leq0$ and $w_{3}^{\pm}\geq0.$ Moreover, we have $$(w_{2}^{\pm})^2+(w_{3}^{\pm})^2=(w_{1}^{\pm})^2-2w_{2}^{\pm}w_{3}^{\pm}.$$ Using this last information we get
\begin{eqnarray*}
|W^{+}|^2&=&(w_{1}^{+})^2+(w_{2}^{+})^2+(w_{3}^{+})^2\\
 &=&2(w_{1}^{+})^2-2w_{2}^{+}w_{3}^{+}.
\end{eqnarray*}
Next, since $ w_{1}^{+}w_{3}^{+}\leq w_{2}^{+}w_{3}^{+}$ and $(w_{1}^{+})^2=-w_{1}^{+}w_{3}^{+}-w_{1}^{+}w_{2}^{+}$ we infer
\begin{eqnarray}
\label{3456}
|W^{+}|^2\leq6(w_{1}^{+})^2.
\end{eqnarray} Similarly, we deduce $|W^{-}|^2\leq6(w_{1}^{-})^2.$ Then we have
\begin{eqnarray*}
\mathcal{W}(g)&\leq&\int_M|W|^2dV_g\\
 &\leq&6\int_M\left(w_{1}^{+}+w_{1}^{-}\right)^2dV_g,
\end{eqnarray*} where we have used that $w_{1}^{+}w_{1}^{-}\geq 0.$ Finally, we use (\ref{[1.6]}) to arrive at
\begin{eqnarray*}
\mathcal{W}(g)&\le &\frac{1}{6}\int_M\left(s-12K_{1}^{\perp}\right)^2dV_g\\
 &=&\frac{1}{6}\mathcal{E}_{1}^{\perp}(g),
\end{eqnarray*} as we wanted to prove.
\end{proof}

In particular, Lemma \ref{keylemma} allows us to deduce
\begin{equation}
\mathcal{W}(M)\leq\frac{1}{6}\mathcal{E}_{1}^{\perp}(M).
\end{equation}

Next, we present a relationship between $\mathcal{E}_{1}^{\perp}(M)$ and ${\rm Vol_{|K^{\perp}|}(M)}.$ More precisely, we have the following proposition.

\begin{proposition}\label{propdesEV}
Let $M^4$ be a 4-dimensional oriented compact manifold. Then
\begin{equation*}
\mathcal{E}_{1}^{\perp}(M)\leq 576{\rm Vol_{|K^{\perp}|}(M)}.
\end{equation*}
\end{proposition}

\begin{proof}
First, we assume that $g$ is a metric on $M^4$ such that $|K^{\perp}(g)|\leq 1.$ We then choose an orthonormal frame $\{e_1, e_2, e_3, e_4\}$  to deduce that the scalar curvature $s_g$ satisfies
\begin{eqnarray*}
s_g &=&2\sum_{i<j}K(e_i,e_j)\\
 &=&2\left(K(e_1,e_2)+K(e_1,e_3)+K(e_1,e_4)+K(e_2,e_3)+K(e_2,e_4)+K(e_3,e_4)\right)\\
 &=&4\left(K^{\perp}(e_1,e_2)+K^{\perp}(e_1,e_3)+K^{\perp}(e_1,e_4)\right).
\end{eqnarray*} Whence it follows that $|s_g|\leq 12.$

On the other hand, we use the standard Cauchy's inequality to infer
\begin{eqnarray}
\frac{1}{6}\mathcal{E}_{1}^{\perp}(g)&=&\frac{1}{6}\int_M\left(s-12K_{1}^{\perp}\right)^2dV_g\nonumber\\
 &\leq&\frac{1}{3}\int_M\left(s^2+144(K_{1}^{\perp})^2\right)dV_g\nonumber\\
 &\leq&\frac{1}{3}\int_M(144+144)dV_g\nonumber\\
 &\leq&96Vol(M,g)\label{desE1vol},
\end{eqnarray} where we have used that $|K^{\perp}(g)|\leq 1$ implies $|K_{1}^{\perp}|\leq 1.$ From here it follows that
\begin{eqnarray*}
\frac{1}{6}\mathcal{E}_{1}^{\perp}(M) \leq 96Vol_{|K^{\perp}|}(M),
\end{eqnarray*} which gives the requested result.
\end{proof}

Our next result has been inspired by Proposition 2.1 in \cite{BCG1} and it plays a fundamental role in the proof of the fourth estimate stated in Theorem \ref{thMinSigEul}.

\begin{proposition}\label{propdesintK1}
Let $M^4$ be a compact 4-dimensional manifold with a metric  $g$ such that $K_{1}^\perp$ is non-positive constant. If $\overline{g}$ is a conformal metric to $g,$ then
\begin{equation*}
\int_{M}|K_{1}^\perp|^2dV_g\leq\int_{M}|\overline{K}_{1}^\perp|^2dV_{\overline{g}},
\end{equation*}
and equality occurs if and only if there exists a constant $c>0$ such that $\overline{g}=cg.$
\end{proposition}
\begin{proof}
We assume that $\overline{g}=e^{2\phi}g\in[g]$ for some function $\phi$ on $M^4.$ From this, we have $$\overline{s}-f_1(\overline{W})=12\overline{K}_{1}^\perp$$ and $$\overline{s}=e^{-2\phi}(-6\Delta\phi-6|\nabla\phi|^2+s),$$ where $f_1(\overline{W})=e^{-2\phi}f_1(W).$ Thus we have
\begin{equation}
\label{propdesintK1eq1}12e^{2\phi}\overline{K}_{1}^\perp=12K_{1}^\perp-6\Delta\phi-6|\nabla\phi|^2.
\end{equation}
Next, we integrate (\ref{propdesintK1eq1}) with respect to the metric $g$ and we then use that  $K_{1}^\perp$ is a non-positive constant as well as Stokes formula to deduce
\begin{eqnarray}
\label{p1}
12\int_{M}|K_{1}^\perp|dV_{g}&=&-12\int_{M}\overline{K}_{1}^\perp e^{2\phi} dV_{g} -6\int_{M}|\nabla \phi|^2 dV_{g}\nonumber\\&\leq& 12\int_{M}e^{2\phi}|\overline{K}_{1}^\perp|dV_{g}.
\end{eqnarray} In particular, the equality holds in (\ref{p1}) if and only if $\phi$ is a constant function.

On the other hand, from Cauchy-Schwarz inequality we have
\begin{eqnarray}
\label{p2}
\left(\int_{M}e^{2\phi}|\overline{K}_{1}^\perp|dV_g\right)^2\leq \left(\int_{M}e^{4\phi}|\overline{K}_{1}^\perp|^2dV_g\right)Vol(M,g)
\end{eqnarray} and since $dV_{\overline{g}}=e^{4\phi}dV_g$ we arrive at

\begin{eqnarray}
\label{p3}
 \frac{1}{Vol(M,g)}\Big(\int_{M}e^{2\phi}|\overline{K}_{1}^\perp|dV_{g}\Big)^2 \leq \int_{M}|\overline{K}_{1}^\perp|^2 dV_{\overline{g}}.
\end{eqnarray}
Now, it suffices to combine (\ref{p1}) and (\ref{p3}) to finish the proof of the proposition.
\end{proof}

Proceeding, we shall investigate the behavior of the invariant  $\mathcal{Y}_{1}^{\perp}(M).$ To do so, we follow the ideas developed in \cite{LeBrunCAG} and  \cite{Itoh} to get the following lemma.

\begin{lemma}\label{lemYg}
Let $M^4$ be a 4-dimensional compact manifold. Then
\begin{equation*}
\inf_{\overline{g}\in[g]}\int_M|\overline{K}_{1}^{\perp}|^2dV_{\overline{g}}=\frac{|\mathcal{Y}_{1}^{\perp}(M,[g])|^2}{144}
\end{equation*} for any conformal class $[g].$
\end{lemma}

\begin{proof}
To begin with, we suppose that $\mathcal{Y}_{1}^{\perp}(M,[g])\geq 0.$ So, we use the Cauchy-Schwarz inequality, for each metric, to obtain
\begin{equation*}
\frac{12}{Vol(M,g)^{\frac{1}{2}}}\int_M
K_{1}^{\perp}dV_g\leq\left(144\int_M|K_{1}^{\perp}|^2dV_g\right)^{\frac{1}{2}},
\end{equation*} with equality if and only if $K_{1}^{\perp}$ is a non-negative constant. Next, taking the infimum over the conformal class of $g$ and using (\ref{defY1}) we arrive at
\begin{eqnarray*}
\mathcal{Y}_{1}^{\perp}(M,[g])\leq\inf_{\overline{g}\in[g]}\left(144\int_M|\overline{K}_{1}^{\perp}|^2dV_{\overline{g}}\right)^{\frac{1}{2}}.
\end{eqnarray*}
On the other hand, as it was observed by Itoh in \cite{Itoh} (see also Proposition 2.2.2 in \cite{listing10}) there is a metric $\overline{g}\in[g]\cap\mathcal{M}_1$  of constant modified scalar curvature. So, for a such metric we have $12\overline{K}_{1}^{\perp}=\mathcal{Y}_{1}^{\perp}(M,[g])\geq 0.$ This allows us to deduce
\begin{equation*}
\inf_{\overline{g}\in[g]}\int_M|\overline{K}_{1}^{\perp}|^2dV_{\overline{g}}=\frac{\mathcal{Y}_{1}^{\perp}(M,[g])^2}{144}
\end{equation*} and this gives our assertion.

Finally, we assume that $\mathcal{Y}_{1}^{\perp}(M,[g])\leq 0.$  We also consider $\overline{g}\in[g]\cap\mathcal{M}_1$ such that $12\overline{K}_{1}^{\perp}=\mathcal{Y}_{1}^{\perp}(M,[g])\leq 0$ is constant and we then apply Proposition \ref{propdesintK1} to infer
\begin{equation*}
\int_M|\overline{K}_{1}^{\perp}|^2dV_{\overline{g}}=\frac{|\mathcal{Y}_{1}^{\perp}(M,[g])|^2}{144}\leq\int_M|{K}_{1}^{\perp}|^2dV_g.
\end{equation*} Moreover, the equality holds if and only if $\overline{g}=cg,$ for some constant $c>0.$ This finishes the proof of the lemma.
\end{proof}

In light of Lemma \ref{lemYg} we deduce the following proposition, which can be compared with Proposition 1 in \cite{LeBrunCAG}.

\begin{proposition}\label{propY1}
Let $M^4$ be a 4-dimensional compact manifold. Then
\begin{equation*}
\inf_{g\in\mathcal{M}}\int_{M}|K_{1}^\perp|^2dV_g=\left\{ \begin{array}{cc}
0\,, & \mbox{ if } \mathcal{Y}_{1}^{\perp}(M)>0;\\
\frac{|\mathcal{Y}_{1}^{\perp}(M)|^2}{144}, & \mbox{ if } \mathcal{Y}_{1}^{\perp}(M)\leq 0. \end{array} \right.
\end{equation*}
\end{proposition}
\begin{proof}
We first assume that $\mathcal{Y}_{1}^{\perp}(M)\leq0,$ then $\mathcal{Y}_{1}^{\perp}(M,[g])\leq0,$ for all conformal class $[g].$ In particular, we can write
\begin{eqnarray*}
\inf_{[g]}|\mathcal{Y}_{1}^{\perp}(M,[g])|^2&=&\left(-\sup_{[g]}\mathcal{Y}_{1}^{\perp}(M,[g])\right)^2\\
 &=&\left(-\mathcal{Y}_{1}^{\perp}(M)\right)^2\\
 &=&|\mathcal{Y}_{1}^{\perp}(M)|^2.
\end{eqnarray*} Combining this information with Lemma \ref{lemYg} we obtain our assertion.

Next, we assume that $\mathcal{Y}_{1}^{\perp}(M)>0,$ we then apply the same argument used by  LeBrun in Proposition 1 of \cite{LeBrunCAG} to conclude the proof of the proposition.
\end{proof}

Also in \cite{LeBrunCAG} it was proved that collapses with bounded scalar curvature is directly relevant to the computation of Yamabe invariants. Based on this result we obtain the following equivalences.

\begin{proposition}\label{thequiv}
Let $M^4$ be a 4-dimensional compact manifold. Then the following statements are equivalent:
\begin{enumerate}
\item ${\rm Vol_{|K_{1}^{\perp}|}(M)}=0;$
\item $\displaystyle{\inf_{g\in\mathcal{M}}\int_M|K_{1}^{\perp}|^2dV_g=0};$
\item $\mathcal{Y}_{1}^{\perp}(M)\geq 0.$
\end{enumerate}
\end{proposition}
\begin{proof}
First, we assume that ${\rm Vol_{|K_{1}^{\perp}|}(M)}=0.$ Then we consider a sequence of metric $\{g_{i}\}_{i\ge 1}$ in $\mathcal{M}_{|K_{1}^{\perp}|}=\{g\in\mathcal{M};\,|K_{1}^{\perp}|\leq 1\}$ such that   $Vol(M,g_i)$ converge to $0$ when $i$ goes to infinity. From this setting, we have $$\int_M|(K_{1}^{\perp})_i|^2dV_{g_i}\leq Vol(M,g_i)$$ and therefore we deduce $$\inf_{g\in\mathcal{M}_{|K_{1}^{\perp}|}}\int_M|K_{1}^{\perp}|^2dV_g=0.$$ On the other hand, since $\mathcal{M}_{|K_{1}^{\perp}|}\subset\mathcal{M}$ we infer
$$0\leq\inf_{g\in\mathcal{M}}\int_M|K_{1}^{\perp}|^2dV_g\leq\inf_{g\in\mathcal{M}_{|K_{1}^{\perp}|}}\int_M|K_{1}^{\perp}|^2dV_g=0$$ and this gives the second assertion.

Next, it suffices to use Proposition \ref{propY1} to prove that the second assertion implies the third one.

Now, we treat of the last case. Indeed, according to \cite{besse} any smooth manifold of dimension $\geq3$ admits metrics of
negative scalar curvature and therefore $M^4$ admits metrics such that $K_{1}^{\perp}$ is negative. Assuming that $\mathcal{Y}_{1}^{\perp}(M)\geq 0$ there exits a sequence of metrics $\{g_i\}$ such that $Vol(M,g_i)=1$ and $(K_{1}^{\perp})_i$ converge to $0$ when $i$ goes to infinity with $(K_{1}^{\perp})_i<0$ for all $i.$ Defining $\overline{g}_i=|(K_{1}^{\perp})_i|_{\infty}g_i$ we obtain $s_{\overline{g}_i}=\frac{1}{|(K_{1}^{\perp})_i|_{\infty}}s_{g_i}.$ From this, a straightforward  computation gives
\begin{eqnarray*}
12(\overline{K}_{1}^{\perp})_i&=&s_{\overline{g}_i}-f_1(\overline{W}_i)=12\frac{(K_{1}^{\perp})_{i}}{|(K_{1}^{\perp})_{i}|_{\infty}}.
\end{eqnarray*}
Therefore, since $|(\overline{K}_{1}^{\perp})_{i}|\le 1$ we get
\begin{eqnarray*}
Vol(M,\overline{g}_i)&=&|(K_{1}^{\perp})_i|_{\infty}^2Vol(M,g_i)\\
 &=&|(K_{1}^{\perp})_i|_{\infty}^2\rightarrow0,
\end{eqnarray*} which gives the desired result.
\end{proof}

One should be emphasized that on standard sphere $\Bbb{S}^4$ we have $\mathcal{Y}_{1}^{\perp}(\mathbb{S}^4)=\mathcal{Y}(\mathbb{S}^4)=8\pi\sqrt{6},$ where $\mathcal{Y}(M)$ stands for the standard Yamabe invariant. Using Proposition  \ref{thequiv} it is easy to see that ${\rm Vol_{|K_{1}^{\perp}|}}(\mathbb{S}^4)=0.$ Combining this information with Theorem \ref{thMinSigEul} we conclude that inequality (\ref{lp1}) can be strict.

\section{Proofs of the main results}
\label{proofs}

\subsection{Proof of Theorem \ref{thMinSigEul}}

\begin{proof}
In order to prove the first estimate we invoke a result due to Gray (cf. Theorem 3.2 in \cite{Gray}) which asserts that on the oriented orthonormal frame $\{e_{1},e_{2},e_{3},e_{4}\}$ the signature $\tau$ of $M^4$ is given by
\begin{eqnarray}
\label{eq1pv1}
\tau (M^4)&=&\frac{1}{6\pi^2}\int_{M}\big\{(K_{12}+K_{34})R_{1234} +(K_{13}+K_{24})R_{1324}+(K_{14}+K_{23})R_{1423}\nonumber\\&&+(R_{1323}-R_{1424})(R_{1314}-R_{2324})-(R_{1232}-R_{1434})(R_{1214}-R_{2343})\nonumber\\&&+(R_{1242}-R_{1343})(R_{1213}-R_{4243})\big\} dV_g,
\end{eqnarray} where $K_{ij}$ is the sectional curvature of the plane generated by $e_i$ and $e_j.$

Now we use Corollary 4.1 in \cite{BG}  to conclude that there exists an orthonormal basis $\{e_{1},e_{2},e_{3},e_{4}\}$ such that the components $R_{1213},$ $R_{1214},$ $R_{1223},$ $R_{1224},$ $R_{1314}$ and $R_{1323}$ all vanish. This information combined with (\ref{eq1pv1}) and (\ref{[1.2]}) provides
\begin{eqnarray}
\label{eq1pv2}
\tau (M^4)&=&\frac{1}{6\pi^2}\int_{M}\big\{2K_{12}^{\perp}R_{1234} +2K_{13}^{\perp}R_{1324}+2K_{14}^{\perp}R_{1423}+R_{1424}R_{2324}\nonumber\\&&-R_{1434}R_{2343}+R_{1343}R_{4243}\big\} dV_g.
\end{eqnarray}

On the other hand, Seaman \cite{Seaman} (see also \cite{berger60}) has proved that
\begin{equation}
\label{sea1}
|R_{ijkl}|\leq\frac{2}{3}(K_{3}^{\perp}-K_{1}^{\perp}).
\end{equation} Hence, applying  Seaman's estimate (\ref{sea1}) in equation (\ref{eq1pv2}) and using that  $|K^{\perp}|\leq 1$ we infer
\begin{eqnarray}
\label{eq1pv3}
6\pi^2|\tau(M^4)|\leq \int_{M}\big\{4(K_{3}^\perp -K_{1}^\perp)+\frac{4}{3}(K_{3}^{\perp}-K_{1}^{\perp})^2 \big\} dV_g
\end{eqnarray} and this implies
$$6\pi^2|\tau(M^4)|\leq \frac{40}{3}Vol(M,g),$$ which gives the first assertion.

Now, we treat of the second statement. Indeed, we recall that Bishop and Goldberg  \cite{BG} showed that for a suitable orthonormal basis $\{e_{1},e_{2},e_{3},e_{4}\},$ the Euler characteristic of $M^4$ can be written as

\begin{equation}
\label{characteristic2}\chi(M)=\frac{1}{4\pi^2}\int_M\left(K_{12}K_{34}+K_{13}K_{24}+K_{14}K_{23}+(R_{1234})^2+(R_{1324})^2+(R_{1423})^2\right)dV_g.
\end{equation}

Therefore, since $ab\leq\left(\frac{a+b}{2}\right)^2$ for all $a,b\in\mathbb{R},$ we infer

\begin{eqnarray}
\label{q1a}
4\pi^2\chi(M)&\leq&\int_{M}\Big(\Big(\frac{K_{12}+K_{34}}{2}\Big)^2+\Big(\frac{K_{13}+K_{24}}{2}\Big)^2+\Big(\frac{K_{14}+K_{23}}{2}\Big)^2 \nonumber\\
&&+(R_{1234})^2+(R_{1324})^2+(R_{1423})^2\Big)dV_g\nonumber\\
 &\leq&\int_{M}\Big( (K_{12}^{\perp})^2+(K_{13}^{\perp})^2+(K_{14}^{\perp})^2+(R_{1234})^2+(R_{1324})^2+(R_{1423})^2\Big) dV_g,
\end{eqnarray} where $K_{ij}^{\perp}$ stands for the biorthogonal curvature of the plane generated by $e_i$ and $e_j.$

From this, we compare (\ref{sea1}) with (\ref{q1a}) to arrive at
\begin{equation*}
4\pi^2\chi(M)\leq\int_M\left((K_{12}^{\perp})^2+(K_{13}^{\perp})^2+(K_{14}^{\perp})^2+\frac{4}{3}(K_{3}^{\perp}-K_{1}^{\perp})^2\right)dV_g.
\end{equation*} Whence, if $|K^{\perp}|\leq1,$ then $$4\pi^2\chi(M)\leq\frac{25}{3}Vol(M,g).$$ From here it follows that $$Vol_{|K^{\perp}|}(M)\geq\frac{12}{25}\pi^2\chi(M)$$ and this gives the second statement.

Proceeding, we assume that $\chi(M)\leq0.$ Since $-\left(\frac{a-b}{2}\right)^2 \leq ab$ for all $a,b\in\mathbb{R},$ we may use (\ref{characteristic2}) to infer $$\int_M\left(-\frac{(K_{12}-K_{34})^2}{4}-\frac{(K_{13}-K_{24})^2}{4}-\frac{(K_{14}-K_{23})^2}{4}\right)dV_g\leq 4\pi^2\chi(M)\leq 0.$$ Therefore, $4\pi^2|\chi(M)|\leq3Vol(M,g)$ and taking the infimum among metrics $g$ such that $|K(g)|\leq1$ we conclude the proof of the third assertion.

Finally, we shall prove the last assertion. To do so, it suffices to prove that $${\rm Vol_{|K_{1}^{\perp}|}(M)}\geq\frac{1}{144}|\mathcal{Y}_{1}^\perp(M)|^2.$$ Indeed, let $g$ be a metric such that $|K_{1}^{\perp}|\leq1$ and assume that $\mathcal{Y}_1^\perp(M)\leq0.$  By Itoh's work \cite{Itoh} (see also Proposition 2.2.2 in \cite{listing10}) there is a metric $\overline{g}\in[g]\cap\mathcal{M}_1$ of constant modified curvature $12\overline{K}_{1}^{\perp}=\mathcal{Y}_{1}^{\perp}(M,[g])\leq \mathcal{Y}_{1}^{\perp}(M)\leq0.$  Since $\overline{K}_{1}^{\perp}$ is a non-positive constant we invoke Proposition \ref{propdesintK1} to arrive at $$\int_{M}|\overline{K}_{1}^{\perp}|^2dV_{\overline{g}}\leq\int_{M}|K_{1}^{\perp}|^2dV_{g}.$$ Therefore, it follows from Proposition \ref{propY1} that
\begin{eqnarray*}
|\mathcal{Y}_{1}^{\perp}(M)|^2&\leq&144\int_M|\overline{K}_{1}^{\perp}|^2dV_{\overline{g}}\\ &\leq&144\int_{M}|K_{1}^{\perp}|^2dV_{g}\\
 &\leq&144 Vol(M,g),
\end{eqnarray*} which finishes the proof of Theorem \ref{thMinSigEul}.
\end{proof}

\subsection{Proof of Theorem \ref{thmB}}

\begin{proof}
To begin with, we recall that the Euler characteristic of $M^4$ can be written as 
\begin{equation}
\label{chi2}
\chi(M)=2-2b_1+b_2,
\end{equation} where $b_i$ is the i-th Betti number of $M^4.$ Since $M^4$ is simply connected we conclude that $M^4$ has positive Euler characteristic. Whence, we may use the second item of Theorem \ref{thMinSigEul} to infer
\begin{equation}
\label{p1p}
{\rm Vol_{|K^{\perp}|}(M)}\geq\frac{12}{25}\pi^2\chi(M).
\end{equation}
Next, our assumption together with (\ref{p1p}) provides $\chi(M)\le 3.$ This information combined with (\ref{chi2}) allows us to deduce that either $b_{2}=0$ or $b_{2}=1.$ Therefore, we invoke Freedman's classification \cite{freedman} (see also \cite{donaldson}) to conclude that either $M^4$ is homeomorphic to $\Bbb{S}^4$ or to $\Bbb{CP}^2,$ as we wanted to prove.
\end{proof}

\subsection{Proof of Theorem \ref{thmincur}}
\begin{proof}
The proof of the first assertion is straightforward. We start recalling that by Chern-Gauss-Bonnet formula, the Euler characteristic satisfies
\begin{equation}
\label{characteristic} 8\pi^2 \chi (M)= \int_{M} \Big(| W^+ |^2 + |W^-|^2 + \frac{s^2}{24} - \frac{1}{2}|\mathring{Ric}|^2\Big) dV_g,
\end{equation} where $\mathring{Ric}$ is the Ricci traceless. Whence it follows that
\begin{eqnarray*}
8\pi^2|\chi(M)|&\leq&\int_M\left(\frac{s^2}{24}+|W|^2+\frac{1}{2}|\mathring{Ric}|^2\right)dV_g\\
 &=&\int_M|Rm|^2dV_g\\
 &\leq&[\mathcal{R}^{\infty}(g)]^2Vol(M,g).
\end{eqnarray*} So, it suffices to take the infimum in $\mathcal{M}_1$ to arrive at $$Mincur(M)\geq2\pi\sqrt{2|\chi(M)|},$$ which proves the first assertion.

Proceeding, the Hirzebruch signature theorem says that
\begin{equation}
\label{signature} 12\pi^2\tau (M) = \int_{M} \Big(| W^+|^2 - |W^-|^2\Big) dV_g
\end{equation} and this gives
\begin{eqnarray*}
12\pi^2|\tau(M)|&\leq&\int_M\left(|W^{+}|^2+|W^{-}|^2\right)dV_g\\
 &\leq&[\mathcal{R}^{\infty}(g)]^2Vol(M,g).
\end{eqnarray*} From this, $Mincur(M)\geq2\pi\sqrt{3|\tau(M)|},$ as claimed.

Now, we treat of the third assertion. Indeed, we have already noted that
\begin{eqnarray*}
[\mathcal{R}^{\infty}(g)]^{2} &\geq& |Rm|^{2}=\frac{s^{2}}{24}+|W|^2 +\frac{1}{2}|Ric -\frac{s}{4}g|^{2}\nonumber\\ &\geq& \frac{s^{2}}{24}+|W|^{2}.
\end{eqnarray*} On integrating the above inequality over $M^4$ and using that $Vol(M,g)=1$ we arrive at
\begin{eqnarray}
\label{kp}
[\mathcal{R}^{\infty}(g)]^{2} &\geq& \int_{M}\frac{s^{2}}{24}dV_{g}+\int_{M}|W|^{2}dV_{g}\nonumber\\&\geq& \inf_{g\in\mathcal{M}}\int_{M}\frac{s^{2}}{24}dV_{g}+12\pi^{2}|\tau(M)|,
\end{eqnarray} where we have used Eq. (\ref{signature}). Whence, we use Proposition 1 in \cite{LeBrunCAG}  as well as (\ref{kp}) to get
\begin{equation}
[\mathcal{R}^{\infty}(g)]^{2} \geq\frac{|\mathcal{Y}(M)|^{2}}{24}+12\pi^{2}\tau(M).
\end{equation} Taking the infimum over $\mathcal{M}_{1}$ we infer $$[Mincur(M)]^{2}\geq 12\pi^2|\tau(M)|+\frac{|\mathcal{Y}(M)|^{2}}{24},$$ as asserted. 

Now, we shall prove the fourth assertion. To do so, let $g$ be a metric on $M^4$ of volume 1. Next, from decomposition of the Riemann curvature tensor we get
\begin{equation}
\label{pq1}
|Rm|^2\geq\frac{s^2}{24}+|W|^2.
\end{equation} Moreover, since $s=12K_{1}^{\perp}+f_1(W)$ we have $$s^2\geq\frac{144(K_{1}^{\perp})^2}{1+\alpha}-\frac{(f_1(W))^2}{\alpha},$$
for all $\alpha>0.$ This information combined with (\ref{pq1}) provides
\begin{equation}
\label{thmincurdes1}|Rm|^2\geq\frac{6(K_{1}^{\perp})^2}{1+\alpha}-\frac{(f_1(W))^2}{24\alpha}+|W|^2,
\end{equation} for all $\alpha>0.$

On the other hand, it is not hard to prove that
\begin{equation}
\label{eqT}
|w_{1}^{\pm}|^2\le \frac{2}{3}|W^{\pm}|^2.
\end{equation} Moreover, the equality holds in (\ref{eqT}) if and only if $w_{3}^{\pm}=w_{2}^{\pm}.$ For sake of completeness we shall sketch it here. Indeed, since $w_1^{+} + w_2^{+} + w_3^{+}
= 0$ we infer
\begin{equation}
\label{i1}
(w^{+}_1)^2=(w^{+}_2)^2+(w^{+}_3)^2+2w^{+}_2w^{+}_3.
\end{equation} Moreover, we have
\begin{equation}
\label{i2}
0\leq(w^{+}_3-w^{+}_2)^2=(w^{+}_3)^2-2w^{+}_2w^{+}_3+(w^{+}_2)^2.
\end{equation} So, we compare (\ref{i1}) with (\ref{i2}) to deduce that $$2[(w^{+}_2)^2+(w^{+}_3)^2]\geq(w^{+}_1)^2.$$ From what follows that $(w_1^{+})^2\leq\frac{2}{3}|W^{+}|^2.$
Similarly, we have $(w_1^{-})^2\leq\frac{2}{3}|W^{-}|^2$ and this proves (\ref{eqT}).

Proceeding, we combine (\ref{eqT}) with (\ref{K1F1}) and (\ref{[1.6]}) to deduce
\begin{eqnarray*}
\left(f_1(W)\right)^2&=&36(w_1^{+}+w_1^{-})^2\\
 &\leq&72[(w_1^{+})^2+(w_1^{-})^2]\\
 &\leq&48|W|^2
\end{eqnarray*} and this implies
\begin{eqnarray}
\label{pq3}
|W|^2-\frac{\left(f_1(W)\right)^2}{24\alpha}&\geq&\frac{\alpha-2}{\alpha}|W|^2.
\end{eqnarray}

Next, we compare (\ref{pq3}) with (\ref{thmincurdes1}) to arrive at

\begin{equation}
\label{pq4}
|Rm|^2\geq\frac{6(K_{1}^{\perp})^2}{1+\alpha}+\frac{\alpha-2}{\alpha}|W|^2.
\end{equation} By setting $\alpha=2$ in (\ref{pq4}) we deduce $|Rm|^2\geq2|K_{1}^{\perp}|^2.$ In particular, we get
\begin{equation}
\label{pq5}
[\mathcal{R}^{\infty}(g)]^2\geq2|K_{1}^{\perp}|^2.
\end{equation} Now, upon integrating of (\ref{pq5}) over $M^4$ we take the
infimum over $\mathcal{M}_1$ to infer
\begin{eqnarray*}
[Mincur(M)]^2&=&\inf_{g\in\mathcal{M}_1}[\mathcal{R}^{\infty}(g)]^2\\
 &\geq&2\inf_{g\in\mathcal{M}_1}\int_M |K_{1}^{\perp}|^2dV_g\\
 &\geq&2\inf_{g\in\mathcal{M}}\int_M |K_{1}^{\perp}|^2dV_g.
\end{eqnarray*}
Finally, it suffices to apply Proposition \ref{propY1} to get $$[Mincur(M)]^2\geq\frac{|\mathcal{Y}_{1}^{\perp}(M)|^2}{72}$$ and this gives the last statement. This finishes the proof of the theorem.
\end{proof}
 
\subsection{Proof of Theorem \ref{thmD}}
\begin{proof}
Suppose that $M^4$ is a simply connected such that $Mincur(M)= 2\pi\sqrt{2|\chi(M)|},$ which tell us that $\chi(M)=2+b_{2}$ and then the Euler characteristic of $M^4$ is positive. This enables us to use (\ref{characteristic}) together with our assumption to deduce

\begin{eqnarray}
\label{kl21}
\int_{M}\frac{s^2}{24}dV_g+\int_{M}|W|^2dV_g-\frac{1}{2}\int_{M}|\mathring{Ric}|^2dV_g&=&8\pi^2\chi(M)\nonumber\\
 &=&[Mincur(M)]^2\nonumber
\end{eqnarray} From this, we then use that there is a metric $g$ on $M^4$ satisfying $Mincur(M)=\mathcal{R}^{\infty}(g)$ to arrive at

\begin{eqnarray}
\label{kl2}
\int_{M}\frac{s^2}{24}dV_g+\int_{M}|W|^2dV_g-\frac{1}{2}\int_{M}|\mathring{Ric}|^2dV_g &=&[\mathcal{R}^{\infty}(g)]^2\nonumber\\
 &\geq&\frac{s^2}{24}+|W|^2+\frac{1}{2}|\mathring{Ric}|^2.
\end{eqnarray} On integrating (\ref{kl2}) over $M^4$ and using that $Vol(M^4,\,g)=1$ we arrive at $$\int_{M}|\mathring{Ric}|^2dV_g\leq0.$$ This forces $M^4$ to be Einstein. In particular, it has constant scalar curvature. Next, substituting these informations in (\ref{kl2}) we have $$\int_{M}|W|^2dV_g\geq|W|^2.$$ From this setting, the function $F:=\int_{M}|W|^2dV_g-|W|^2\geq0$ and then $$0\leq\int_{M}F dV_g=\int_{M}|W|^2dV_g-\int_M|W|^2dV_g=0,$$ where we have used that $Vol(M^4,\,g)=1.$ This allows us to conclude that $|W|$ is constant, which finishes the proof of the first part of the theorem

Finally, we assume that $M^4$ has nonnegative sectional curvature. Now, it suffices to use Corollary 1.1 in \cite{Costa} to conclude that $M^4$ is isometric to $\Bbb{S}^4,$ $\Bbb{CP}^2$ or $\Bbb{S}^2\times \Bbb{S}^2,$ as we wanted to prove.
\end{proof}

\begin{acknowledgement}
The E. Ribeiro Jr would like to thank the Department of Mathematics - Lehigh University, where part of this work was carried out. He is grateful to H.-D. Cao for the warm hospitality and his constant encouragement. Moreover, the authors want to thank D. Kotschick and G. Yun for many helpful comments that benefited the presentation of this article.
\end{acknowledgement}

\end{document}